\documentclass[12pt,letterpaper,reqno]{amsart} 
\usepackage[top=2cm,bottom=2cm,left=1.85cm,right=1.85cm]{geometry}

\usepackage[utf8]{inputenc}
\usepackage{amssymb,amsmath,amsfonts,eucal,mathrsfs,amsthm,xcolor} 
\usepackage{mathtools}
\usepackage{enumitem}
\usepackage{amsrefs}
\usepackage{stmaryrd}
\usepackage{hyperref}
\hypersetup{
    colorlinks=true,
    filecolor=magenta,      
    linkcolor={blue!50!black},
    citecolor={red!50!black},
    urlcolor={blue!80!black}
}

\usepackage{dutchcal}

\newtheorem{theorem}{Theorem}
\newtheorem{proposition}{Proposition}
\newtheorem{lemma}[proposition]{Lemma}

\newtheorem{corollary}[proposition]{Corollary}

\newtheorem{remark}[proposition]{Remark}
\newtheorem{remarks}[proposition]{Remarks}

\newtheorem*{question}{Question}

\theoremstyle{definition}

\newcommand{\R}{\mathbb{R}}

\newcommand{\Sf}{\mathbb{S}}

\newcommand{\Ric}{\mbox{Ric}}

\newcommand{\End}{\mbox{End}}
\newcommand{\Hom}{\mbox{Hom}}

\newcommand{\Tp}{T_A^{[p]}}
\newcommand{\tr}{\mathrm{tr}}

\def\bea{\begin{eqnarray*} }
\def\eea{\end{eqnarray*} }

\def\beq{\begin{equation}}

\def\B{\mathcal{B}}
\def\<{{\langle}}
\def\>{{\rangle}}

\def\n{\nabla}

\def\a{\alpha}

\def\be{\begin{equation} }
\def\ee{\end{equation} }

\begin{document}

\title{On $Q$-convex hypersurfaces in Riemannian manifolds}
\author{Giulio Colombo and Christos-Raent Onti}
\date{\today}
\maketitle

\renewcommand{\thefootnote}{\fnsymbol{footnote}} 
\renewcommand{\thefootnote}{\arabic{footnote}} 

\renewcommand{\thefootnote}{\fnsymbol{footnote}} 
\renewcommand{\thefootnote}{\arabic{footnote}} 

\begin{abstract}
We prove that any closed, convex hypersurface in an $(n+1)$-dimensional Riemannian manifold with $\lceil \frac{n}{2} \rceil$-positive curvature operator is a rational homology sphere with finite fundamental group. The same conclusion holds for any $\lceil \frac{n}{2} \rceil$-convex hypersurface, provided that the mean curvature satisfies a sharp pinching condition.
Both results follow from more general vanishing and estimation theorems for the Betti numbers of closed $q$-convex immersed hypersurfaces in $(n+1)$-dimensional Riemannian manifolds, under a lower bound on the average of the smallest $(n-p)$ eigenvalues of the curvature operator. 
\end{abstract}

\section{Introduction}
Let $M$ be an $(n+1)$-dimensional Riemannian manifold, and let $f \colon \Sigma \to  M$ be an immersed hypersurface.
In this paper, we always assume that $\Sigma$ is closed, connected, oriented and $f$ is two-sided (that is, there exists a global non-vanishing unit normal vector field $\nu$).
We say that $f$ is $q$-convex (resp. strictly $q$-convex) for some integer $1 \leq q \leq n$ if the sum of its smallest $q$ principal curvatures is nonnegative (resp. positive) at each point $x \in \Sigma$, with respect to $\nu$. 
Observe that $q$-convexity implies $p$-convexity for all $p \geq q$ and non-negativity $H\geq 0$ of the mean curvature in the direction of $\nu$. 

The special case $q = 1$ corresponds to convexity (that is, all principal curvatures are nonnegative). 
In fact, when the ambient space is the Euclidean space, we know that $\Sigma$ is diffeomorphic to the sphere $\Sf^n$ by the work of Hadamard \cite{ha1897}, Chern-Lashof \cite{cl58}, Heijenoort \cite{vh52}, and Sacksteder \cite{sac60} (see also do Carmo-Lima \cite{dl69}). The same holds in nonflat space forms due to do Carmo and Warner \cite{docawa} and also in Hadamard manifolds (that is, complete, simply connected Riemannian manifolds with nonpositive sectional curvature) due to Alexander \cite{alex77}. 
For further discussion regarding this case, we refer the reader to the survey of Lima \cite{lima} and the references therein.

The next case of $2$-convexity has also attracted considerable attention. In particular, the topological classification in Euclidean space $\R^{n+1}$, $n \geq 3$, was established by Huisken and Sinestrari \cite{hs09}, who showed that such hypersurfaces are diffeomorphic either to the sphere $\Sf^n$ or to a connected sum of $\Sf^{n-1} \times \Sf^1$. Their approach relied on the mean curvature flow with surgery together with the fact that $2$-convexity is preserved along the flow. As pointed out by Brendle and Huisken \cite{bh17}, this may no longer be true for the mean curvature flow in general ambient Riemannian manifolds. 
To overcome this difficulty, they introduced a fully nonlinear flow preserving, among other things, strict $2$-convexity in general ambient spaces. Their work implies that compact $(n+1)$-dimensional Riemannian manifolds, $n \geq 3$, satisfying
$
\bar{R}_{1313} + \bar{R}_{2323} \ge 0,
$
and having non-empty strictly $2$-convex boundary are diffeomorphic to a $1$-handlebody. The case $n = 2$ was also treated by Brendle and Huisken \cite{bh16,bh18}.

More generally, for $q$-convexity with $1 \leq q \leq n-1$, Sha \cite{sha86} (see also Wu \cite{wu87}) proved that every compact $n$-dimensional Riemannian manifold with nonnegative sectional curvature and strictly smooth $q$-convex boundary has the homotopy type of a CW-complex of dimension at most $q-1$. As a consequence, if $q \leq \frac{n}{2}$, then the $i$-th real homology group of the boundary vanishes for all $q \leq i \leq n-q$.

In the present paper, we apply the Bochner technique to provide new vanishing and estimation results for the Betti numbers of 
$q$-convex immersed hypersurfaces in general ambient Riemannian manifolds, under a lower bound on the average of the smallest $(n-p)$ eigenvalues 
of the curvature operator. We recall that the curvature operator of a Riemannian manifold is called $\ell$-positive (resp. $\ell$-nonnegative) 
if the sum of its smallest $\ell$ eigenvalues is positive  (resp. nonnegative). In this sense, our results may be viewed as an extrinsic analogue of the work of Petersen and Wink \cite{pw21}.

Our first main theorem is the following:

\begin{theorem}\label{mainA1}
Let $n\geq 3$ and $1\leq q\leq p\leq \lfloor \frac{n}{2} \rfloor$ be integers. If $M$ is an $(n+1)$-dimensional 
Riemannian manifold with $(n-p)$-nonnegative curvature operator and $f\colon \Sigma \to M$ a closed, 
connected, oriented, $q$-convex immersed hypersurface, then the following hold:

\begin{enumerate}[topsep=1pt,itemsep=1pt,partopsep=1ex,parsep=0.5ex,leftmargin=*, label={\rm(\arabic*)}, align=left, labelsep=0.4em]
\item The Betti numbers of $\Sigma$ satisfy
$
b_i(\Sigma)=b_{n-i}(\Sigma)\leq \binom{n}{i}, 
$
for each $q\leq i\leq p$.
\item If either the curvature operator of $M$ is $(n-p)$-positive at some point $x\in f(\Sigma)$ or $f$ is strictly $q$-convex at some point, 
then 
\[
b_q(\Sigma)=\cdots=b_p(\Sigma)=0 \quad \text{and}\quad b_{n-p}(\Sigma)=\cdots=b_{n-q}(\Sigma)=0.
\]
Moreover, if $q=1$, then $\Sigma$ admits a Riemannian metric of positive Ricci curvature and has finite fundamental group. 
In particular, if $n=3$, then it is diffeomorphic to a spherical space form.
\item If $b_q(\Sigma)>0$, then every harmonic 
$q$-form is parallel. In particular, $\Sigma$ supports a nontrivial parallel $q$-form. 
\end{enumerate}
\end{theorem}

As an immediate consequence, we get 

\begin{theorem}\label{mainA0}
Let $n\geq 3$ and $1\leq q\leq p\leq \lfloor \frac{n}{2} \rfloor$ be integers. If $M$ is an $(n+1)$-dimensional 
Riemannian manifold with $(n-p)$-positive curvature operator and $f\colon \Sigma\to M$ a closed, 
connected, oriented, $q$-convex immersed hypersurface, then the Betti numbers of $\Sigma$ satisfy
\[
b_q(\Sigma)=\cdots=b_p(\Sigma)=0 \quad \text{and}\quad b_{n-p}(\Sigma)=\cdots=b_{n-q}(\Sigma)=0.
\]
Moreover, if $q=1$, then $\Sigma$ admits a Riemannian metric of positive Ricci curvature and has finite fundamental group. 
In particular, if $n=3$, then it is diffeomorphic to a spherical space form.
\end{theorem}

For $n \ge 3$, the class of $(n+1)$-dimensional Riemannian manifolds with $\ell$-positive curvature operator $(3\leq \ell \leq n)$ is different from the class of manifolds satisfying $\overline{R}_{1313} + \overline{R}_{2323} \geq 0$ for all orthonormal sets $\{e_1,e_2,e_3\}$, that is, having nonnegative second intermediate Ricci curvature $\Ric_2$, defined for unit vectors $v\in TM$ by
\[
    \Ric_2(v) = \inf\{ \sec(v,e_1) + \sec(v,e_2) : \{e_1,e_2\} \text{ orthonormal set in } v^\bot \}
\]

Let us present some further interesting consequences of Theorem \ref{mainA1}. For example, it yields the following ``sphere theorem'' for hypersurfaces.

\begin{corollary}
Let $M$ be an $(n+1)$-dimensional Riemannian manifold with $n \geq 3$ and $\lceil \frac{n}{2} \rceil$-nonnegative curvature operator.
Let $f \colon \Sigma \to M$ be a closed, connected, oriented, convex immersed hypersurface.
If either the curvature operator of $M$ is $\lceil \frac{n}{2} \rceil$-positive at some point $x \in f(\Sigma)$, or $f$ is strictly convex at some point, then $\Sigma$ is a rational homology sphere, admits a Riemannian metric with positive Ricci curvature, and has finite fundamental group. In particular, if $n = 3$, then $\Sigma$ is diffeomorphic to a spherical space form.
\end{corollary}

Of course the conclusions of Theorems \ref{mainA1}-\ref{mainA0} do apply also in case the ambient manifold $M$ satisfies the relevant curvature assumption only in a neighbourhood of $f(\Sigma)$. So, for compact Riemannian manifolds with non-empty boundary we get:

\begin{corollary}\label{extsh}
Let $n \geq 3$, and let $M$ be a compact $(n+1)$-dimensional Riemannian manifold with non-empty smooth boundary $\partial M$. 
If the curvature operator of $M$ is $\lceil \frac{n}{2} \rceil$-positive in a neighbourhood of $\partial M$ and $\partial M$ is $q$-convex for some $1\leq q\leq \lfloor \frac{n}{2} \rfloor$, 
then each connected component of $\partial M=M_1\sqcup\dots\sqcup M_k$ 
satisfies
\[
b_q(M_j)=\cdots=b_{n-q}(M_j)=0\quad \text{for all}\quad 1\leq j\leq k.
\]
In particular, if $\partial M$ is convex then each connected component of $\partial M$ is a rational homology sphere with finite fundamental group.
\end{corollary}

When the average of the smallest $\ell$ eigenvalues of the curvature operator of the ambient manifold is bounded from below by a negative constant,
then we prove the following:

\begin{theorem}\label{mainB1}
Fix integers $n\geq 3$ and $1\leq q\leq p\leq \lfloor \frac{n}{2} \rfloor$ and real numbers $c<0$ and $D>0$.
Let $M$ be an $(n+1)$-dimensional Riemannian manifold such that 
\[
\frac{\lambda_1+\cdots+\lambda_{n-p}}{n-p}\geq c,
\] 
where $\lambda_1\leq \cdots\leq \lambda_{\binom{n+1}{2}}$ denote the 
eigenvalues of its curvature operator. Let $f\colon \Sigma\to M$ be a closed, connected, oriented, $q$-convex immersed hypersurface.
If ${\rm Ric}(\Sigma)\geq (n-1)c$ and ${\rm diam}(\Sigma)\leq D$, then there exists a constant $C(n,cD^2)>0$
such that 
$$
b_i(\Sigma)=b_{n-i}(\Sigma)\leq \binom ni \cdot \exp\left( C(n,c D^2)\cdot \sqrt{-cD^2i(n-i)}\right),
$$
for each $q\leq i\leq p$. Moreover, there exists $\varepsilon(n,i)>0$ such that if $cD^2\geq -\varepsilon(n,i)$, then 
$
b_i(\Sigma)=b_{n-i}(\Sigma)\leq \binom{n}{i}.
$ 
\end{theorem}

We note that when $f$ is convex, the hypothesis on ${\rm Ric}(\Sigma)$ is automatic.

\medskip 

Theorems \ref{mainA1} and \ref{mainB1} naturally raise the

\begin{question}
Can we say anything about $b_i(\Sigma)$ with $i<q\leq \frac{n}{2}$? What about when $q>\frac{n}{2}$?
\end{question}

This is exactly what the next two theorems establish. 

\begin{theorem}\label{mainA2}
Let $n\geq 3$, $1\leq p\leq \lfloor \frac{n}{2} \rfloor$ and $1\leq q\leq n-1$ be integers.
Let $M$ be an $(n+1)$-dimensional Riemannian manifold with $(n-p)$-positive 
curvature operator, and let $f\colon \Sigma\to M$ be a closed, connected, oriented, 
$q$-convex immersed hypersurface. Set
\be\label{constc}
c\coloneqq\min_{x\in f(\Sigma)}\left\{\frac{\lambda_1(x)+\cdots+\lambda_{n-p}(x)}{n-p}\right\},
\ee
where 
$
\lambda_1\leq \cdots\leq \lambda_{\binom{n+1}{2}}
$
denote the eigenvalues of the curvature operator of $M$, and 
\be\label{beta}
\beta_q\coloneqq
\begin{cases}
q-1, & \text{if} \qquad q\leq \frac{n}{2}\\[1.5mm]
n-q,  & \text{if} \qquad q>\frac{n}{2}
\end{cases} 
\ee
If for some integer $1\leq \ell \leq \min\{p,\beta_q\}$ the (normalized) mean curvature satisfies 
\be\label{qpinched1}
H\leq \frac{n-q}{n}\sqrt{\frac{\ell c}{q-\ell}},
\ee
then the following hold:
\medskip
\begin{enumerate}[topsep=1pt,itemsep=1pt,partopsep=1ex,parsep=0.5ex,leftmargin=*, label={\rm(\arabic*)}, align=left, labelsep=0.4em]
\item 
$\displaystyle
    b_\ell(\Sigma) = b_{n-\ell}(\Sigma) \leq \binom{n}{\ell} \quad \text{and} \quad b_i(\Sigma)=b_{n-i}(\Sigma) = 0 \quad \text{for all} \quad \ell < i\leq \min\{p,\beta_q\}.
$
\item If strict inequality holds at some point, then 
\[
b_i(\Sigma)=b_{n-i}(\Sigma)=0, \quad \text{for all}\quad  \ell \leq i\leq \min\{p,\beta_q\}.
\]
Moreover, if $\ell=1$, then $\Sigma$ admits a Riemannian metric of positive Ricci curvature, and has finite fundamental group. 
In particular, if $n=3$, then it is diffeomorphic to a spherical space form.
\item If $b_\ell(\Sigma)>0$, then equality holds in \eqref{qpinched1} everywhere on $\Sigma$ and every harmonic 
$\ell$-form is parallel. In particular, $\Sigma$ supports a nontrivial parallel $\ell$-form. 
\end{enumerate}
\end{theorem}

We emphasize that inequality \eqref{qpinched1} is sharp; see Remarks \ref{remsph}(3). Observe also that Theorem \ref{mainA2} complements Theorem \ref{mainA0}. As an immediate consequence of Theorem \ref{mainA2}, we obtain the following ``sphere theorem'' for pinched hypersurfaces.

\begin{corollary}\label{corsh}
Let $n\geq 3$ and let $M$ be an $(n+1)$-dimensional Riemannian manifold with $\lceil \frac{n}{2} \rceil$-positive 
curvature operator. If $f\colon \Sigma \to M$ is a closed, connected, oriented, $\lceil \frac{n}{2} \rceil$-convex immersed 
hypersurface such that 
\be\label{sharp}
H\leq \frac{\lfloor \frac{n}{2} \rfloor}{n\sqrt{\lceil \frac{n}{2} \rceil-1}} \sqrt{c},
\ee
where $c$ is defined by \eqref{constc} for $p=\lfloor n/2\rfloor$, and the inequality is strict at some point, then $\Sigma$ is a rational homology sphere, 
admits a Riemannian metric of positive Ricci curvature, and has finite fundamental group. 
In particular, if $n=3$, then it is diffeomorphic to a spherical space form.
\end{corollary}

We emphasize once again that inequality \eqref{sharp} is sharp; see Remarks \ref{remsph}(4). When the average of the smallest $(n-p)$ eigenvalues of the curvature operator of the ambient Riemannian manifold is bounded from below by a non-positive constant, we obtain the following result.

\begin{theorem}\label{mainB2}
Fix integers $n\geq 3,\; 1\leq p\leq \lfloor \frac{n}{2} \rfloor,\; 1\leq q\leq n-1$ and real numbers $c\leq 0$ and $D>0$.
Let $M$ be an $(n+1)$-dimensional Riemannian manifold such that 
\[
\frac{\lambda_1+\cdots+\lambda_{n-p}}{n-p}\geq c,
\] 
where $\lambda_1\leq \cdots\leq \lambda_{\binom{n+1}{2}}$ denote the 
eigenvalues of its curvature operator. Let $f\colon \Sigma\to M$ be a closed, connected, oriented, $q$-convex immersed hypersurface
with ${\rm diam}(\Sigma)\leq D$. Set $\beta_q$ as in \eqref{beta} and let
$$
    \kappa_\ell=\min_{x\in \Sigma} \left(c -\frac{q-\ell}{\ell}\left(\frac{n}{n-q}\right)^2 H^2\right)
$$
for each $1\leq \ell\leq \min\{p,\beta_q\}$. Then there exists a constant $C(n,\kappa_1 D^2)>0$ such that
\begin{equation} \label{b_ell_in}
    b_\ell(\Sigma)=b_{n-\ell}(\Sigma)\leq \binom{n}{\ell} \cdot \exp\left( C(n,\kappa_1 D^2)\cdot \sqrt{-\kappa_\ell D^2 \ell(n-\ell)}\right)
\end{equation}
for each $1\leq \ell\leq \min\{p,\beta_q\}$. Moreover, there exists $\varepsilon(n,\ell)>0$ such that if $\kappa_1 D^2\geq -\varepsilon(n,\ell)$, then
$
b_\ell(\Sigma)=b_{n-\ell}(\Sigma)\leq \binom{n}{\ell}.
$

\end{theorem}

\begin{remark}
    If $q\leq p$ (in which case we have $\beta_q = q-1$ and therefore $\min\{p,\beta_q\} = q-1 < p$) then the statement of Theorem \ref{mainB2} actually holds also for $\ell$ in the range $q\leq\ell \leq p$, with just $c$ instead of $\kappa_\ell$ on the RHS of \eqref{b_ell_in}.
\end{remark}

If inequality \eqref{qpinched1} is not satisfied for some $1\leq \ell \leq \min\{p,\beta_q\}$ in Theorem \ref{mainA2}, then one can still derive upper bounds for the corresponding Betti number in the spirit of Theorem \ref{mainB2}.

The structure of the paper is as follows. In Sections \ref{auxres} and \ref{bwoper}, we provide the necessary background material and establish several auxiliary results. The proofs of the main theorems are given in Section \ref{mainproofs}. Finally, in Section \ref{amb-sphere}, we examine the special case in which the ambient space is the unit sphere.

\section{An auxiliary result}\label{auxres}

We start with some algebraic preliminaries inspired by Savo \cite{Savo14}.
Let $V$ be a real $n$-dimensional vector space, $n\geq 3$, equipped with a positive definite 
inner product denoted by $\langle\cdot,\cdot\rangle$. We denote by $\mathrm{End}(V)$ 
the set of all endomorphisms of $V$ and by $\Lambda^p V^*$ the $\binom np$-dimensional 
real vector space, defined as the $p$-th exterior power of the dual vector space $V^*=\Hom(V,\R)$ of $V$. 
To each self-adjoint $A\in \End(V)$, we associate the endomorphism $T_A^{[p]}\in\End(\Lambda^p V^*)$ defined by 
$$
T_A^{[p]}=(\tr A)A^{[p]}-A^{[p]}\circ A^{[p]},
$$
where $A^{[p]}\in \End(\Lambda^p V^*)$ is given by 
$$
A^{[p]}\omega(v_1,\dots,v_p)=\sum_{i=1}^p \omega(v_1,\dots,Av_i,\dots,v_p),
$$
with $\omega\in \Lambda^p V^*$ and $v_1,\dots,v_p\in V$. The endomorphism $T_A^{[p]}$ is self-adjoint 
with respect to the natural inner product $\<\cdot, \cdot \>$ in $\Lambda^p V^*$. 
For every integer $1\leq p\leq n$, we consider the set $I_p$ of $p$-multi-indices 
$
I_p=\left\{\{i_1,\dots,i_p\} : 1\leq i_1<\cdots<i_p\leq n\right\}.
$
Let $A\in \mathrm{End}(V)$ be self-adjoint with eigenvalues $k_1\leq k_2\leq \cdots\leq k_n$. 
Before we proceed, we recall the following elementary fact from linear algebra:
\begin{lemma}\label{basic_lem}
For any integer $1\leq p\leq n$, the following holds:
    \[
       \sum_{i=1}^p \langle Av_i,v_i\rangle\geq  k_1 + \dots + k_p 
    \]
for all orthonormal sets $\{v_1,\dots,v_p\}\subseteq V$.
\end{lemma}

\noindent For each $a=\{i_1,\dots,i_p\}\in I_p$, the associated $p$-{\it algebraic curvature} $K_a$ of $A$ is
the number given by
$
K_a= k_{i_1}+\cdots+k_{i_p}.
$
It follows from \cite[Lemma 2]{Savo14} that the $\binom np$ eigenvalues $\lambda_a(T_A^{[p]})$ of $\Tp$ 
are given by 
\begin{equation}\label{lambda}
\lambda_a(T_A^{[p]})=K_a K_{\star a},\; \; a=\{i_1,\dots,i_p\}\in I_p,
\end{equation}
where $\star a \in I_{n-p}$ is defined by $\star a=\{1,\dots,n\}\smallsetminus\{i_1,\dots,i_p\}$. We say that $A$ is 
$q$-nonnegative if $k_1+\cdots+k_q\geq 0$.

\begin{lemma}\label{prop}
Let $A\in \End(V)$ be self-adjoint and $q$-nonnegative for some integer $1\leq q\leq n-1$. 
Then, for each integer $1\leq p\leq \min\{q,n-q\}$ the lowest eigenvalue 
of $T_A^{[p]}\in \End(\Lambda^p V^*)$ satisfies
\begin{equation} \label{T_q-conv}
    \min_{a\in I_p} \lambda_a(T_A^{[p]})  \geq - \frac{(n-p)(q-p)}{(n-q)^2} (\tr A)^2 
\end{equation}
and inequality is strict in \eqref{T_q-conv} if $A$ is $q$-positive. Moreover, if $A\neq 0$ 
and equality holds in \eqref{T_q-conv} then the eigenvalues $k_1\leq\dots\leq k_n$ of $A$ satisfy
\[
\begin{cases}
k_1 + \dots + k_p = - \frac{q-p}{n-q} \tr A \leq 0 \quad \text{and} \quad k_{p+1} = \dots = k_n = \frac{\tr A}{n-q} \geq 0, & \text{when} \quad p<q \\[2mm]
    k_1 + \dots + k_p=0, &  \text{when} \quad p=q
\end{cases}
\]
\end{lemma}

\begin{proof}
Let $k_1\leq\dots\leq k_n$ denote the eigenvalues of $A$. We claim that 
\[
\min_{a\in I_p} \lambda_a(T_A^{[p]})=\lambda_{\{1,\dots,p\}}.
\]
Indeed, the function $t\mapsto t(\tr A-t)$ is concave, so the minimum value 
of $\lambda_a$ is attained when $a$ either minimizes or maximizes $K_a$ 
among all $p$-subsets of $\{1,\dots,n\}$. Due to the ordering of the $k_i$'s, 
this happens when $a=\{1,\dots,p\}$ or $a=\{n-p+1,\dots,n\}$, respectively, 
so we have to show that
    \begin{equation} \label{lamcomp}
        \lambda_{\{n-p+1,\dots,n\}} \geq \lambda_{\{1,\dots,p\}}.
    \end{equation}
    We compute
    \[
        \lambda_{\{1,\dots,p\}} = \sum_{i=1}^p k_i \sum_{i=p+1}^n k_i, \qquad \lambda_{\{n-p+1,\dots,n\}} = \sum_{i=1}^{n-p} k_i \sum_{i=n-p+1}^n k_i \, .
    \]
    Note that $p\leq \frac{n}{2}$ by assumption $p\leq\min\{q,n-q\}$. If $n$ is even and $p=\frac{n}{2}$ then clearly $\lambda_{\{1,\dots,p\}} = \lambda_{\{n-p+1,\dots,n\}}$, so let us suppose that $p<\frac{n}{2}$. In this case we have $p<n-p$, so the sets $\{1,\dots,p\}$, $\{p+1,\dots,n-p\}$ and $\{n-p+1,\dots,n\}$ define a partition of $\{1,\dots,n\}$ and setting
    \[
        B = \sum_{i=1}^p k_i, \qquad C = \sum_{i=p+1}^{n-p} k_i, \qquad D = \sum_{i=n-p+1}^n k_i
    \]
    we have
    \[
        \lambda_{\{1,\dots,p\}} = B(C+D), \qquad \lambda_{\{n-p+1,\dots,n\}} = D(B+C) .
    \]
    Note that $B\leq D$ since the sums defining them contain the same number of terms and the terms appearing in the definition of $B$ are all less or equal than those appearing in the definition of $D$. Moreover, since $p\leq n-q$ we have $q\leq n-p$, so $B+C\geq 0$ by $q$-nonnegativity of $A$ and therefore $C\geq 0$. (Indeed, if it were $C<0$ then we would have $k_{p+1} < 0$, so $k_i<0$ for all $i\leq p$ implying $B<0$ and then $B+C<0$, contradiction.) In view of this, we have
    \[
        \lambda_{\{n-p+1,\dots,n\}} - \lambda_{\{1,\dots,p\}} = C(D - B) \geq 0
    \]
    and this proves \eqref{lamcomp}.

    If $\lambda_{\{1,\dots,p\}} \geq 0$ then \eqref{T_q-conv} is trivial, so let us suppose that $\lambda_{\{1,\dots,p\}} < 0$. Note that
    \[
        K_{\{1,\dots,p\}} + K_{\{p+1,\dots,n\}} = \tr A \geq 0
    \]
    by $q$-nonnegativity of $A$ and that $K_{\{1,\dots,p\}} = B \leq C+D = K_{\{p+1,\dots,n\}}$ by what we already observed. Since $\lambda_{\{1,\dots,p\}} = B(C+D)$, the assumption $\lambda_{\{1,\dots,p\}} < 0$ implies in particular that
    $
        B = K_{\{1,\dots,p\}} < 0 \, .
    $
    We observe that
    \[
        0 \leq k_q \leq \frac{\tr A}{n-q} \, .
    \]
    Indeed, since $A$ is $q$-non-negative we have
    \[
        0 \leq \sum_{i=1}^q k_i \leq qk_q,
    \]
    and
    \begin{equation} \label{Teq1}
        (n-q) k_q \leq \sum_{i=q+1}^n k_i = \tr A - \sum_{i=1}^q k_i \leq \tr A \, .
    \end{equation}
    Using this, we estimate
    \begin{equation} \label{Teq2}
        0 \leq \sum_{i=1}^q k_i = \sum_{i=1}^p k_i + \sum_{j=p+1}^q k_j \leq \sum_{i=1}^p k_i + (q-p) k_q \leq \sum_{i=1}^p k_i + \frac{q-p}{n-q} \tr A
    \end{equation}
    that is,
    \[
        |K_{\{1,\dots,p\}}| = - \sum_{i=1}^p k_i \leq \frac{q-p}{n-q} \tr A
    \]
    and then
    \begin{equation} \label{Teq3}
        |\lambda_{\{1,\dots,p\}}| = |K_{\{1,\dots,p\}} | (\tr A + |K_{\{1,\dots,p\}}|) \leq \frac{(n-p)(q-p)}{(n-q)^2} (\tr A)^2
    \end{equation}
    as desired. If $A\neq 0,\; p<q$ and equality holds in \eqref{T_q-conv}, then $T_A^{[p]}$ must have at least one negative eigenvalue, so $\lambda_{\{1,\dots,p\}} < 0$ 
    indeed, and all inequalities in \eqref{Teq1}, \eqref{Teq2} and \eqref{Teq3} must be satisfied with equality sign. 
    In particular, $k_1+\dots+k_q = 0$ and $k_{p+1} = \dots = k_q = \dots = k_n$ implying
    \begin{equation} \label{Teq4}
        \tr A = \sum_{i=q+1}^n k_i = (n-q)k_q \qquad \text{and} \qquad \sum_{i=1}^p k_i = -\sum_{i=p+1}^q k_i = -(q-p)k_q
    \end{equation}
    as claimed. If $A\neq 0,\; p=q$ and equality holds in \eqref{T_q-conv}, then it is immediate that $k_1+\dots+k_q = 0$.

 If $A$ is $q$-positive then we have strict inequality in the last of \eqref{Teq1}. If $p<q$ this implies that the inequality is also strict in the last of \eqref{Teq2} and in \eqref{Teq3}. On the other hand, if $p=q$ the LHS of \eqref{T_q-conv} is zero while the RHS is strictly positive since $\lambda_a > 0$ for all $a\in I_p = I_q$ by strict $q$-convexity (note that if the case $p=q$ occurs then necessarily $q\leq n-q$ due to assumption $p\leq\min\{q,n-q\}$, hence strict $q$-convexity implies $K_a > 0$ both for all $a\in I_q$ and for all $a\in I_{n-q}$). In both cases, \eqref{T_q-conv} holds with strict inequality sign.
\end{proof}

\section{The Bochner-Weitzenb\"ock operator}\label{bwoper}

Let $M$ be a closed, connected and oriented Riemannian manifold of 
dimension $n$. For each integer $1\leq p\leq n-1$, the Hodge Laplacian 
acting on $p$-forms is defined by 
$$
\Delta=dd^*+d^* d: \Omega^p(M)\to \Omega^p(M),
$$
where $d$ and $d^*$ are the differential and the co-differential operators, respectively. 
For every $p$-form $\omega$ the Bochner-Weitzenb\"ock formula states that the Hodge Laplacian is given by
\be\label{boch.form}
\Delta\omega=\n^*\n\omega+\B^{[p]}\omega,
\ee
where $\n^*\n$ is the connection Laplacian and 
$$
\B^{[p]}\colon \Omega^p(M)\to \Omega^p(M)
$$
is a certain symmetric and self-adjoint endomorphism on the bundle of $p$-forms, called the {\it Bochner-Weitzenb\"ock operator}.

\begin{proposition}\label{pbochner}
Let $M$ be a closed, connected and oriented $n$-dimensional Riemannian manifold and let $1\leq p\leq \frac{n}{2}$.
The following assertions hold:
\begin{enumerate}[topsep=2pt,itemsep=2pt,partopsep=2ex,parsep=0.5ex,leftmargin=*, label={\rm(\arabic*)}, align=left, labelsep=0.4em]
\item If $\B^{[p]}\geq 0$, then $b_p(M)\leq \binom{n}{p}.$
\item If $\B^{[p]}\geq 0$, and the strict inequality holds at some point, then the de Rham cohomology satisfies
$$H^p(M,\R)=H^{n-p}(M,\R)=0.$$ 
Moreover, if $p=1$, then $M$ admits a Riemannian metric of positive Ricci curvature, and has finite fundamental group. In particular, if $n=3$, then it is diffeomorphic to a spherical space form.
\item If $\B^{[p]}\geq 0$ and $H^p(M,\R)\neq0$, then every harmonic $p$-form is parallel. 
In particular, $M$ supports a nontrivial parallel $p$-form.
\end{enumerate}
\end{proposition}

\begin{proof}
We argue for (2) when $p=1$. Recall that $$\<\B^{[1]}\omega,\omega\>=\Ric(\omega^\#,\omega^\#)$$ where $\omega^\#$ denotes the dual vector field corresponding to the $1$-form $\omega$. Therefore, in this case we have that the Ricci curvature is everywhere non-negative and at some point is strictly positive.  
Then, it follows from Aubin \cite{A} that $M$ carries a metric with positive Ricci curvature and the Bonnet-Myers theorem implies that the fundamental group of $M$ is finite. 
In particular, if $n=3$ then the result follows from \cite{ha82}. For the rest of the proofs see {\cite[Proposition 3]{Savo14} and \cite[Proposition 15]{ov22}}. \qed
\end{proof}

\medskip

The following is a special case of Theorem 1.9 in \cite{pw21} adapted to our needs.

\begin{proposition}\label{peterthm}
Let $n\geq 3,\; \kappa<0$ and $D>0$, and let $M$ be a closed $n$-dimensional 
Riemannian manifold with ${\rm{Ric}}(M)\geq (n-1) \kappa$ and ${\rm diam}(M)\leq D$. 
Assume that for some integer $1\leq p\leq \frac{n}{2}$ there exists $\delta>0$ such that 
$
\B^{[p]}\geq \delta \kappa. 
$
Then, there exists a constant $C(n,\kappa D^2)>0$ such that 
the dimension of the kernel of the Hodge Laplacian 
$$
{\rm ker}(\Delta)=\{\omega\in \Omega^p(M)\; :\; \Delta\omega=\n^*\n \omega + \B^{[p]}\omega=0\}
$$
is bounded by 
$$
\binom np \cdot \exp\left( C(n,\kappa D^2)\cdot \sqrt{-\kappa \delta D^2}\right).
$$
Moreover, there exists $\varepsilon(n,\delta)>0$ such that if $\kappa D^2 \geq -\varepsilon(n,\delta)$, then 
$
\dim {\rm ker}(\Delta)\leq \binom np.
$
\end{proposition}

\subsection{The relation with the Curvature Operator} 

Let $M$ be an $n$-dimensional Riemannian manifold 
with Riemannian connection $\n$ and curvature tensor $R$ 
defined by 
$$R(X,Y)=[\n_X,\n_Y]-\n_{[X,Y]},\; X,Y\in TM.$$
For each $X,Y\in TM$ we can view $R(X,Y)$ as a skew-symmetric endomorphism on $TM$. 
Next, recall that the curvature operator is the self-adjoint operator $\mathfrak{R}\colon \Lambda^2TM\to \Lambda^2TM$ defined by 
$$\<R(X,Y)Z,W\>=\<\mathfrak{R}(X\wedge Y),W\wedge Z\>,$$
where $\Lambda^2TM$ inherits naturally the inner product coming from the Riemannian metric of $M$. Every 
$X\wedge Y\in \Lambda^2TM$ acts on $TM$ as a skew symmetric endomorphism via 
$$
(X\wedge Y)Z=\<X,Z\>Y-\<Y,Z\>X,\; Z\in TM.
$$
Therefore, each element of $\Lambda^2TM$ can be viewed as a skew symmetric endomorphism of $TM$. 
For this reason we write $\mathfrak{so}(TM)=\Lambda^2TM$ and endow $\mathfrak{so}(TM)$ with the 
inner product that comes from $\Lambda^2TM$. Now, every element of $\mathfrak{so}(TM)$ induces a derivative on 
$p$-forms in the following way: if $\omega\in \Omega^p(M)$ and $L\in \mathfrak{so}(TM)$ then 
\[
(L\omega)(X_1,\dots,X_p)=-\sum_{i=1}^p \omega(X_1,\dots,LX_i,\dots,X_p).
\]
To each $p$-form $\omega$ we assign a tensor $\hat \omega$ with values in $\Lambda^2TM$ defined by:
\[
\<L,\hat\omega(X_1,\dots,X_p)\>=(L\omega)(X_1,\dots,X_p)\quad \text{for all}\quad L\in \mathfrak{so}(TM)=\Lambda^2TM.
\]
It follows from \cite[Lemma 9.4.9, page 354]{p16} that the Bochner-Weitzenb\"ock operator and the curvature operator of $M$ 
are related via the formula
\begin{equation}\label{R}
\<\B^{[p]}\omega,\phi\>=\<\mathfrak{R}(\hat\omega),\hat\phi\> \quad \text{for all}\quad \omega,\phi\in\Omega^p(M).
\end{equation}

\subsection{The splitting of the Bochner-Weitzenb\"ock Operator}

Let $f\colon M\to \tilde M, n\geq 3,$ be an isometric immersion into a Riemannian manifold $\tilde M$ of dimension $n+m$.
The second fundamental form $\a_f$ is viewed as a section of the vector bundle $\mathrm{Hom}(TM\times TM,N_f M)$, where 
$N_f M$ is the normal bundle. For each unit normal vector field $\xi\in \Gamma(N_fM)$, the associated 
shape operator $A_{\xi}$ is given by 
$$
\<A_{\xi} X,Y\>=\langle \alpha_f(X,Y),\xi\rangle,\: \: X,Y\in TM.
$$
The (normalized) \textit{mean curvature vector field} of $f$ is $\mathcal H=(\mathrm{tr}\, \alpha_f)/n$, where $\mathrm{tr}$ 
means taking the trace. In particular, when $m=1$ the \textit{mean curvature} $nH$ is the trace of the shape operator and when $M$ is orientable 
we always choose an orientation such that $H\geq 0$.
Let $R$ and $\tilde R$ be the curvature tensors of $M$ and $\tilde M$ respectively. Then from the 
Gauss equation we have 
\[
\<R(X,Y)Z,W\>=\<\tilde R(X,Y)Z,W\>+\<R_{\rm ext}(X,Y)Z,W\> \quad X,Y,Z,W\in TM,
\]
where 
\[
\<R_{\rm ext}(X,Y)Z,W\>=\<\alpha_f(X,W),\alpha_f(Y,Z)\>-\<\alpha_f(X,Z),\alpha_f(Y,W)\>.
\] 
Accordingly, we can split the curvature operator $\mathfrak{R}$ of $M$ as the sum of two self-adjoint operators
acting on $\Lambda^2TM$ by
\begin{equation}\label{Rsplit}
\mathfrak{R}=\mathfrak{R}_{\rm res}+\mathfrak{R}_{\rm ext}
\end{equation}
where $\mathfrak{R}_{\rm res}$ and $\mathfrak{R}_{\rm ext}$ are defined by
\[
\begin{aligned}
\langle \mathfrak{R}_{\rm res}(X\wedge Y), W\wedge Z\rangle &= \langle \tilde{\mathfrak{R}}(X\wedge Y), W\wedge Z\rangle, \\
\langle \mathfrak{R}_{\rm ext}(X\wedge Y), W\wedge Z\rangle &= \<\alpha_f(X,W),\alpha_f(Y,Z)\>-\<\alpha_f(X,Z),\alpha_f(Y,W)\>.
\end{aligned}
\]
Let $\{\xi_1,\dots,\xi_m\}$ be an orthonormal frame of the normal bundle. Then 
$$\mathfrak{R}_{\rm ext}=\sum_{\alpha=1}^m \mathfrak{R}_{\rm ext}^{(\alpha)},$$ where 
\[
\langle \mathfrak{R}_{\rm ext}^{(\alpha)}(X\wedge Y), W\wedge Z\rangle = 
\<A_{\xi_\alpha}(X),W\>\<A_{\xi_\alpha}(Y),Z\>-\<A_{\xi_\alpha}(X),Z\>\<A_{\xi_\alpha}(Y),W\>
\]
It follows from \eqref{R} and \eqref{Rsplit} that the Bochner-Weitzenb\"ock operator $\B^{[p]}$ of $M$ acting on $p$-forms splits as
\begin{equation}\label{Bsplit}
\B^{[p]}=\B^{[p]}_{{\rm res}}+\B^{[p]}_{{\rm ext}},
\end{equation}
where
\begin{equation}\label{Bsplit2}
\begin{aligned}
\langle \B^{[p]}_{{\rm res}}\omega,\phi\rangle &= \<\mathfrak{R}_{\rm res}(\hat\omega),\hat\phi\>
=\<\tilde{\mathfrak{R}}(\hat\omega),\hat\phi\>,  \\
\langle \B^{[p]}_{{\rm ext}}\omega,\phi\rangle &= \<\mathfrak{R}_{\rm ext}(\hat\omega),\hat\phi\>
=\sum_{\alpha=1}^m \<\mathfrak{R}_{\rm ext}^{(\alpha)}(\hat\omega),\hat\phi\>.
\end{aligned}
\end{equation}
It is not hard to see that the splitting in \eqref{Bsplit} is in fact the one already obtained 
by Savo \cite[Theorem 1]{Savo14}, where he followed the approach of Petersen \cite{p98} via 
the formalism of Clifford multiplication. Indeed, for each unit vector field $\xi\in \Gamma(N_fM)$, we define 
the endomorphism 
$$
T_{A_\xi}^{[p]}= (\tr A_\xi)A_\xi^{[p]}-A_\xi^{[p]}\circ A_\xi^{[p]},
$$
where 
$
A_\xi^{[p]}\colon \Omega^p(M)\to \Omega^p(M)
$
is the self-adjoint extension of the shape operator $A_\xi$ in the direction $\xi$ defined by 
$$
A_\xi^{[p]}\omega(X_1,\dots,X_p)=\sum_{i=1}^p \omega(X_1,\dots,A_\xi X_i,\dots,X_p),\; X_1,\dots,X_p\in TM.
$$
Then $\B^{[p]}_{\rm{ext}}$ is given by 
\begin{equation}\label{Bext}
\B^{[p]}_{{\rm ext}}=\sum_{\alpha=1}^m T_{A_{\xi_\alpha}}^{[p]}.
\end{equation}

The proof of \eqref{Bext} follows directly from the following algebraic result.

\begin{lemma}
Let $A$ be a self-adjoint endomorphism of $TM$ and consider the associated self-adjoint ``curvature operator" 
$\mathfrak{R}_A$ acting on $\Lambda^2TM$ determined uniquely by the formula 
\[
\langle \mathfrak{R}_A(X\wedge Y), W\wedge Z\rangle = \<A(X),W\>\<A(Y),Z\>-\<A(X),Z\>\<A(Y),W\>
\]
for all $X,Y,Z,W\in TM$. Then, the self-adjoint operator $T_A^{[p]}\colon \Omega^p(M)\to \Omega^p(M)$ defined by 
\begin{equation}\label{Tau}
\langle T_A^{[p]}\omega,\phi\rangle=\<\mathfrak{R}_A(\hat\omega),\hat\phi\>
\end{equation}
can be written as 
\begin{equation}\label{TauSa}
T_A^{[p]}=(\tr A)A^{[p]}-A^{[p]}\circ A^{[p]},
\end{equation}
where $A^{[p]}$ is the self-adjoint extension of $A$ to $\Omega^p(M)$.
\end{lemma}

\begin{proof}
Let $\{e_1,\dots,e_n\}$ be an orthonormal basis diagonalizing $A$ with corresponding 
eigenvalues $k_1,\dots,k_n$. Denote by $\{\theta_1,\dots,\theta_n\}$ its dual basis and 
for each $a=\{i_1,\dots,i_p\}\in I_p$ consider the $p$-form 
$$
\Theta_a=\theta_{i_1}\wedge\dots\wedge \theta_{i_p}.
$$
Savo proved in \cite{Savo14} that $\Theta_a$ diagonalizes $T_A^{[p]}$ given by \eqref{TauSa} 
with corresponding eigenvalues $K_aK_{\star a}$, where 
$$
K_a=\sum_{i\in a} k_i \quad \text{and}\quad K_{\star a}=\sum_{i\in \star a} k_i. 
$$ 
A direct computation shows that 
$T_A^{[p]}$ defined by \eqref{Tau} also shares the same property. Indeed,
\begin{equation}\nonumber
\begin{aligned}
\langle T_A^{[p]}\Theta_a,\Theta_a\rangle &=\<\mathfrak{R}_A(\hat\Theta_a),\hat\Theta_a\> \\
&=\sum_{\ell_1,\dots,\ell_p} \<\mathfrak{R}_A(\hat\Theta_a)(e_{\ell_1},\dots,e_{\ell_p}),\hat\Theta_a(e_{\ell_1},\dots,e_{\ell_p})\> \\
&=\frac{1}{2}\sum_{\ell_1,\dots,\ell_p} \sum_{i\neq j} k_ik_j \<\hat\Theta_a(e_{\ell_1},\dots,e_{\ell_p}),e_i\wedge e_j\>^2 \\
&=\frac{1}{2}\sum_{\ell_1,\dots,\ell_p} \sum_{i\neq j} k_ik_j \Big(\big((e_i\wedge e_j)\Theta_a\big)(e_{\ell_1},\dots,e_{\ell_p})\Big)^2 \\
&=\frac{1}{2}\sum_{\ell_1,\dots,\ell_p} \sum_{i\neq j} k_ik_j \Big(\sum_k \Theta_a(e_{\ell_1},\dots,(e_i\wedge e_j)e_{\ell_k},\dots,e_{\ell_p})\Big)^2 \\
&=\frac{1}{2}\sum_{\ell_1,\dots,\ell_p} \sum_{i\neq j} k_ik_j \Big(\sum_k \Theta_a(e_{\ell_1},\dots,\<e_i,e_{\ell_k}\>e_j-\<e_j,e_{\ell_k}\>e_i,\dots,e_{\ell_p})\Big)^2 \\
&=\sum_{i\in a,\; j\in\star a} k_ik_j =K_aK_{\star a}. \\
\end{aligned}
\end{equation}
It is also very easy to check that $\langle T_A^{[p]}\Theta_a,\Theta_b\rangle =0$ for $a\neq b$, and this completes the proof.
\qed
\end{proof}

\medskip 

\noindent Next, we extend \cite[Theorem 1]{Savo14} by applying results obtained 
by Petersen and Wink in \cite{pw21}.

\begin{lemma}\label{Bres}
Let $\tilde{\lambda}_1\leq \cdots\leq \tilde{\lambda}_{\binom{n+m}{2}}$ denote the 
eigenvalues of $\tilde{\mathfrak{R}}$ and assume that for some real number $c$ 
and some integer $1\leq p\leq \lfloor \frac{n}{2}\rfloor$, the eigenvalues satisfy 
\begin{equation}\label{lambda_t}
\frac{\tilde{\lambda}_1+\cdots+\tilde{\lambda}_{n-p}}{n-p}\geq c.
\end{equation}
Then 
\[
\<\B^{[\ell]}_{{\rm res}}\omega,\omega\>\geq c\ell (n-\ell)|\omega|^2,\quad \text{for all}\quad 1\leq \ell \leq p
\] 
and non-zero $\ell$-forms $\omega$. 
\end{lemma}

\begin{proof}
We only need to prove it for $\ell=p$. To this end, 
we apply \cite[Lemma 2.1]{pw21} for $\mathfrak{R}_{\rm res}$. Indeed, 
if $\lambda_1\leq \cdots\leq \lambda_{\binom{n}{2}}$ are its eigenvalues, then it 
follows from \eqref{Bsplit2}, \eqref{lambda_t} and Lemma \ref{basic_lem} that 
\[
\frac{\lambda_1+\cdots+\lambda_{n-p}}{n-p}\geq 
\frac{\tilde{\lambda}_1+\cdots+\tilde{\lambda}_{n-p}}{n-p}\geq c.
\]
Since from \cite[Lemma 2.2(b) and Proposition 2.5(a)]{pw21} every $p$-form $\omega$ satisfies 
\[
|L\omega|^2\leq \frac{1}{n-p}|\hat\omega|^2|L|^2\quad \text{for all}\quad L\in\mathfrak{so}(TM)=\Lambda^2TM,
\]
the proof follows. \qed
\end{proof}

\medskip 

The following result plays a crucial role in the proofs of our main results.

\begin{proposition}\label{propb}
Fix integers $n\geq 3$, $1\leq p\leq \lfloor \frac{n}{2} \rfloor$ and $1\leq q\leq n-1$.
Let $M$ be an $(n+1)$-dimensional Riemannian manifold such that 
\begin{equation}\label{teigenlb}
\frac{\lambda_1+\cdots+\lambda_{n-p}}{n-p}\geq c,
\end{equation}
for some real constant $c$, where $\lambda_1\leq \cdots\leq \lambda_{\binom{n+1}{2}}$ 
denote the eigenvalues of its curvature operator. If $f\colon \Sigma\to M$ is a closed, connected, oriented, 
$q$-convex immersed hypersurface, then its Bochner-Weitzenb\"ock operator satisfies pointwise the inequality
\be\label{boch2}
\mathop{\min_{\omega\in \Omega^\ell(\Sigma)}}_{\|\omega\|=1}\<\B^{[\ell]}\omega,\omega\>
\geq \ell(n-\ell)\left(c -\frac{q-\ell}{\ell}\left(\frac{n}{n-q}\right)^2 H^2\right),
\ee
for all $1\leq \ell \leq \min\{p,\min\{q,n-q\}\}$. Moreover, \eqref{boch2} is strict if 
either \eqref{teigenlb} is strict or $f$ is strictly $q$-convex. 
Furthermore, if equality holds in \eqref{boch2} at some point $x\in \Sigma$ and the shape operator $A(x)\neq 0$, 
then the principal curvatures $k_1(x)\leq\dots\leq k_n(x)$ satisfy
\[
\begin{cases}
k_1(x) + \dots + k_\ell(x) = - \frac{q-\ell}{n-q} n H(x) \leq 0 
\quad \text{and} \quad 
k_{\ell+1}(x) = \dots = k_n(x) = \frac{nH(x)}{n-q} \geq 0, & \text{when} \quad \ell<q \\[2mm]
k_1(x) + \dots + k_q(x) =0, &  \text{when} \quad \ell=q
\end{cases}
\]

\end{proposition}

\begin{proof}
Follows from \eqref{Bsplit}, \eqref{Bext}, Lemma \ref{prop} and Lemma \ref{Bres}. 
\qed
\end{proof}

\section{Proofs of the main Theorems}\label{mainproofs} 

We first recall that $i$-convexity follows from $q$-convexity, when $i>q$. In the 
following proofs this fact is used essentially by reapplying Proposition \ref{propb}, where is necessary.

\medskip

\noindent{\it Proof of Theorem \ref{mainA1}}. 
From Proposition \ref{propb} we have $\B^{[i]}\geq 0$ for all integers $q\leq i\leq p$. 
The result follows from Proposition \ref{pbochner}(1)-(3) and Poincar\'{e} duality. \qed

\medskip

\noindent{\it Proof of Theorem \ref{mainB1}}. 
From Proposition \ref{propb} we have 
$\B^{[i]}\geq i(n-i)c$ for all integers $q\leq i\leq p$. 
The result follows from Proposition \ref{peterthm} and the fact that $b_i(\Sigma)$ 
is equal to the dimension of the kernel of the Hodge Laplacian. \qed

\medskip

\noindent{\it Proof of Theorem \ref{mainA2}}. 
Assume $q\leq \frac{n}{2}$. From condition \eqref{qpinched1} we have
\[
    c \geq \frac{q-\ell}{\ell} \left(\frac{n}{n-q}\right)^2 H^2
\]
hence $\B^{[\ell]} \geq 0$ by Proposition \ref{propb}. Moreover, since $c>0$ and $t \mapsto \frac{t}{q-t}$ is strictly increasing for $t \in [0,q)$ we also have
\[
    H < \frac{n-q}{n}\sqrt{\frac{ic}{q-i}} \qquad \text{and therefore} \qquad c > \frac{q-i}{i} \left(\frac{n}{n-q}\right)^2 H^2
\]
for all integers $\ell < i \leq \min\{p,q-1\}$, hence
$\B^{[i]}>0$ for those values of $i$. Cases (1)-(3) follow by applying Proposition \ref{pbochner}(1)-(3). To see that equality holds everywhere in \eqref{qpinched1} in case (3),
consider a non-zero $\ell$-form $\omega$ that lies in the kernel of the Hodge Laplacian. Then, from \eqref{boch.form} we obtain
\[
0=|\n\omega|^2+\<\B^{[\ell]}\omega,\omega\>,
\]
which implies that $\<\B^{[\ell]}\omega,\omega\>=0$. 
The desired equality now follows by using \eqref{boch2} and \eqref{qpinched1}. 
Assume $q> \frac{n}{2}$. From Proposition \ref{propb} and \eqref{qpinched1} we have that  $\B^{[\ell]}\geq 0$ and
$\B^{[i]} > 0$ for all integers $\ell < i\leq \min\{p,n-q\}$. 
The cases (1)-(3) are proved similarly as before. 
\qed

\medskip

\noindent{\it Proof of Theorem \ref{mainB2}}.  
Assume $q \leq \frac{n}{2}$. It follows from Proposition \ref{propb} that
$\B^{[\ell]}\geq \ell(n-\ell)\kappa_\ell$, for each $1\leq \ell\leq \min\{p,q-1\}$, 
where 
\[
\kappa_\ell=\min_{x\in \Sigma}\left(c -\frac{q-\ell}{\ell}\left(\frac{n}{n-q}\right)^2 H^2\right).
\]
If moreover $q\leq p$, then from Proposition \ref{propb} we have 
$\B^{[i]}\geq i(n-i)c$ for all $q \leq i\leq p$. 
The result follows from Proposition \ref{peterthm}, noting that $\Ric\geq(n-1)\kappa_1$ (recall that $\Ric(X,X) = \langle\B^{[1]}X^\flat,X^\flat\rangle$ for all vector fields $X$), and Poincar\'{e} duality. On the other hand, if $q> \frac{n}{2}$, then from Proposition \ref{propb} we have
$\B^{[\ell]}\geq \ell(n-\ell)\kappa_\ell$ 
for each $1\leq \ell\leq \min\{p,n-q\}$, 
where $\kappa_\ell$ is as before. The result again follows from 
Proposition \ref{peterthm} and Poincar\'{e} duality. 
\qed

\section{The case of the sphere}\label{amb-sphere}

In this final section we discuss the special case where the ambient 
space is the unit sphere $\mathbb{S}^{n+1}$. 
Recall the standard immersion of a torus 
$$
\mathbb T^n_p(r)=\mathbb{S}^p(r) \times \mathbb{S}^{n-p}(\sqrt{1-r^2})  
$$
into $\mathbb{S}^{n+1}$, where $\mathbb{S}^p(r)$ denotes the $p$-dimensional sphere of radius $r<1$. 
The principal curvatures are $-\sqrt{1-r^2}/r$ and $r/\sqrt{1-r^2}$ of multiplicity $p$ and $n-p$, respectively. 
A direct computation gives that 
\be\label{mean}
H=\frac{nr^2-p}{nr\sqrt{1-r^2}}
\ee
with $r^2\geq p/n$. We note that for $1\leq q\leq n-1$ and $1\leq p\leq \beta_q$ (where, $\beta_q$ is as in \eqref{beta}),
the torus $\mathbb T^n_p(r)$ is $q$-convex, when $r^2\geq p/q$.

\begin{theorem}\label{sphere}
Let $f\colon \Sigma\to  \Sf^{n+1},  n\geq 3$, be a closed, connected, oriented $q$-convex immersed hypersurface for some integer $2\leq q\leq n-1$. 
Set $\beta_q$ as in \eqref{beta}.  If for some integer $1\leq p\leq \beta_q$ the mean curvature satisfies 
\be \label{qpinched2}
H\leq \frac{n-q}{n}\sqrt{\frac{p}{q-p}},
\ee
then
\medskip
\begin{enumerate}[topsep=1pt,itemsep=1pt,partopsep=1ex,parsep=0.5ex,leftmargin=*, label={\rm(\arabic*)}, align=left, labelsep=0.4em]

 \item 
$\displaystyle
    b_p(\Sigma) = b_{n-p}(\Sigma) \leq \binom{n}{p} \quad \text{and} \quad b_i(\Sigma)=b_{n-i}(\Sigma) = 0 \quad \text{for all} \quad p< i\leq \beta_q.
$
\item If strict inequality holds at some point, then 
$$
b_p(\Sigma)=\cdots=b_{\beta_q}(\Sigma)=0 \quad \text{and}\quad  b_{n-\beta_q}(\Sigma)=\cdots=b_{n-p}(\Sigma)=0.
$$
Moreover, if $p=1$, then $\Sigma$ admits a Riemannian metric of positive Ricci curvature and has finite fundamental group. 
\item If $b_p(\Sigma)>0$, then $f(\Sigma)$ is isometric to the torus $\mathbb T^n_p(\sqrt{p/q})$.
\end{enumerate}
\end{theorem}

\begin{proof}
If for some integer $1\leq p\leq \beta_q$ inequality \eqref{qpinched2} is satisfied, then cases
(1) and (2) follow directly from Theorem \ref{mainA2}(1),(2), for $c=1$. If $b_p(\Sigma)>0$ then from Theorem \ref{mainA2}(3)
we get equality in \eqref{qpinched2} everywhere and the result follows from \cite[Theorem 4]{Savo14}. This concludes also case (3). \qed
\end{proof}


For the particular case $p=1$ and assuming strict $2$-convexity we get an even stronger conclusion:

\begin{corollary}\label{cor2con}
Any closed, connected, oriented, strictly $2$-convex embedded hypersurface $f\colon \Sigma\to  \Sf^{n+1},  n\geq 4$, with
$$H\leq \frac{n-2}{n}$$ is diffeomorphic to the sphere $\Sf^n$.
\end{corollary}

\begin{proof}
It follows from \cite[Corollary 1.3]{bh17} that $\Sigma$ is the boundary of a $1$-handlebody, that is, $\Sigma$ is diffeomorphic either to $\Sf^n$ or to a finite connected sum of $g>0$ copies of $\Sf^1 \times \Sf^{n-1}$. In the latter case we would have $b_1(\Sigma) = g > 0$. However, by Proposition \ref{propb} the strict $2$-convexity assumption and inequality $H\leq\frac{n-2}{n}$ imply that $\B^{[1]}$ is strictly positive, so $b_1(\Sigma) = 0$ and we conclude that $\Sigma$ is diffeomorphic to $\Sf^n$.
 \qed
\end{proof}

\begin{remarks}\label{remsph}\ \ 
\smallskip
\rm{
\begin{enumerate}[topsep=1pt,itemsep=1pt,partopsep=1ex,parsep=0.5ex,leftmargin=*, label={\rm(\arabic*)}, align=left, labelsep=0.2em]
 \item Corollary \ref{cor2con} is sharp since $\mathbb T^n_1(1/\sqrt{2})$ is $2$-convex with $H=\frac{n-2}{n}$, but not strictly $2$-convex.
 \item An easy example satisfying \eqref{qpinched2} with strict inequality is $\mathbb T^n_{p-1}(\sqrt{(p-1)/q})$, 
 provided that $p>1$. 
 \item Inequality \eqref{qpinched1} is sharp since $\mathbb T^n_\ell(\sqrt{\ell/q})$ is $q$-convex with $H=\frac{n-q}{n}\sqrt{\frac{\ell}{q-\ell}}$.
 \item Corollary \ref{corsh} is sharp since $\mathbb T^n_1(1/\sqrt{\lceil \frac{n}{2} \rceil})$ is $\lceil \frac{n}{2} \rceil$-convex with $H=\frac{\lfloor \frac{n}{2} \rfloor}{n\sqrt{\lceil \frac{n}{2} \rceil-1}}$.
\end{enumerate} 
}
\end{remarks}

\section*{Statements and Declarations}

\noindent {\bf Conflict of interest} The authors have no relevant financial or non-financial interests to disclose.

\medskip 

\noindent {\bf Data availability} 
Data sharing not applicable to this article as no datasets were
generated or analysed during the current study.

\medskip

\noindent {\bf Funding}
 The authors declare that no funds, grants, or other support were received during the preparation of this manuscript.

\medskip

\noindent {\bf AI Declaration}
In preparing this manuscript, the authors used ChatGPT for proofreading and to check clarity. No new content was generated and all suggested edits were reviewed by the authors.

\vspace{10mm}

\noindent Giulio Colombo\\
Dipartimento di Matematica ``Federigo Enriques"\\
Universit\'{a} degli Studi di Milano\\
Via C. Saldini 50, 20133 Milano (Italy)\\
e-mail: giulio.colombo@unimi.it
\bigskip
\bigskip

\noindent Christos-Raent Onti\\
Department of Mathematics and Statistics\\
University of Cyprus\\
1678, Nicosia -- Cyprus\\
e-mail: onti.christos-raent@ucy.ac.cy
\bigskip

\end{document}